\documentclass{amsart}
\usepackage[toc,page]{appendix}
\usepackage[headheight=12.1pt]{geometry}
\usepackage[all,cmtip]{xy}
\usepackage{amsmath}
\usepackage{amsfonts}
\usepackage{amssymb}
\usepackage{mathrsfs}   
\usepackage{amsthm}
\usepackage{xcolor}
\usepackage{cite}
\usepackage{stmaryrd}
\usepackage{graphicx}
\usepackage{pgf}
\usepackage{eepic}
\usepackage{pstricks}
\usepackage{epsfig}
\usepackage{fancyhdr}
\usepackage{tikz,caption,float}
\usepackage[backref=page]{hyperref}

\usepackage{amsbsy}

\hypersetup{
     colorlinks   = true,
     citecolor    = red,
     linkcolor = blue,
     urlcolor = purple
}
    
  \newcommand\imCMsym[4][\mathord]{%
  \DeclareFontFamily{U} {#2}{}
  \DeclareFontShape{U}{#2}{m}{n}{
    <-6> #25
    <6-7> #26
    <7-8> #27
    <8-9> #28
    <9-10> #29
    <10-12> #210
    <12-> #212}{}
  \DeclareSymbolFont{CM#2} {U} {#2}{m}{n}
  \DeclareMathSymbol{#4}{#1}{CM#2}{#3}
}
\newcommand\alsoimCMsym[4][\mathord]{\DeclareMathSymbol{#4}{#1}{CM#2}{#3}}

\imCMsym{cmmi}{124}{\CMjmath}
\imCMsym[\mathop]{cmsy}{113}{\CMamalg}
\imCMsym[\mathop]{cmex}{96}{\CMcoprod}
\alsoimCMsym[\mathop]{cmex}{97}{\CMbigcoprod}

\theoremstyle{plain}
\newtheorem*{theoremu}{Theorem}
\newtheorem{theorem}{Theorem}[section]
\newtheorem{theoremn}{Theorem}

\newtheorem{proposition}[theorem]{Proposition}

\newtheorem{corollary}[theorem]{Corollary}

\newtheorem*{corollaryu}{Corollary}
\newtheorem{lemma}[theorem]{Lemma}

\theoremstyle{definition}
\newtheorem*{definitionu}{Definition}

\newtheorem{definition}[theorem]{Definition}

\theoremstyle{remark}
\newtheorem{remark}[theorem]{Remark}
\newtheorem*{remarku}{Remark}

\newcommand{\N}{{\mathbb N}}

\newcommand{\Q}{{\mathbb Q}}

\newcommand{\F}{{\mathbb F}}

\renewcommand{\P}{{\mathbb P}}

\newcommand{\isomto}{\overset{\sim}{\rightarrow}}

\newcommand{\lser}[1]{(\!(#1)\!)}
\newcommand{\pow}[1]{\llbracket #1 \rrbracket}

\newcommand{\spec}[1]{\mathrm{Spec}\left(#1\right)}

\newcommand{\spf}[1]{\mathrm{Spf}\left(#1\right)}

\newcommand{\norm}[1]{\left\vert#1\right\vert}

\newcommand{\pn}{(\varphi,\nabla)}

\newcommand{\cris}{\mathrm{cris}}
\newcommand{\rig}{\mathrm{rig}}

\newcommand{\et}{\mathrm{\acute{e}t}}

\usepackage{euscript}

\usepackage{relsize}
\usepackage[bbgreekl]{mathbbol}
\usepackage{amsfonts}
\DeclareSymbolFontAlphabet{\mathbb}{AMSb} 
\DeclareSymbolFontAlphabet{\mathbbl}{bbold}
 
\theoremstyle{definition}
\newtheorem*{AssumptionStar}{Assumption $(\star)$}
\newtheorem*{AssumptionStarSep}{Assumption $(\star^{\rm s})$}
\SetSymbolFont{stmry}{bold}{U}{stmry}{m}{n}
\usepackage{bm}
\usepackage{listings}
\usepackage{mathrsfs}
\DeclareMathAlphabet{\mathscrbf}{OMS}{mdugm}{b}{n}

\title{Good reduction of K3 Surfaces in equicharacteristic $\bm{p}$}

\author{Bruno Chiarellotto}
\address[Chiarellotto]{Dipartimento di Matematica ``Tullio Levi-Civita''\\ Universit\`a degli Studi di Padova\\ Via Trieste 63 \\35121 Padova\\ Italy}
\email{chiarbru@math.unipd.it}

\author{Christopher Lazda}
       \address[Lazda]{Universiteit van Amsterdam\\ Korteweg--de\thinspace Vries Institute for mathematics\\  P.O. Box 94248 \\1090 GE\\ Amsterdam\\ the Netherlands}
       \email{c.d.lazda@uva.nl}
       
\author{Christian Liedtke}
\address[Liedtke]{TU M\"unchen, Zentrum Mathematik - M11, Boltzmannstr. 3, 85748 Garching bei M\"unchen, Germany}
\email{liedtke@ma.tum.de}

\begin{document}

\begin{abstract} 
We show that for smooth and proper varieties over local fields with no non-trivial vector fields, good reduction descends over purely inseparable extensions. 
We use this to extend the N\'eron--Ogg--Shafarevich criterion for K3 surfaces of \cite{CLL17} to the equicharacteristic $p>0$ case.
\end{abstract}

\maketitle 

\tableofcontents

\section*{Introduction}

Let $\mathcal{O}_K$ be an excellent Henselian DVR with perfect residue field $k$ and fraction field $K$. Let $K^\mathrm{s}$ be a separable closure of $K$ and let $G_K=\mathrm{Gal}(K^{\rm s}/K)$ be the corresponding absolute Galois group. Then, the residue field $\bar k$ of $K^{\rm s}$ is an algebraic closure of $k$ and we let $G_k$ denote the corresponding absolute Galois group. Let $X$ be a K3 surface over $K$.

\begin{definitionu} A model for $X$ over $\mathcal{O}_K$ is a flat and proper morphism $\mathcal{X}\rightarrow \spec{\mathcal{O}_K}$ of algebraic spaces together with an isomorphism $X\overset{\cong}{\longrightarrow} \mathcal{X}_K$.
\end{definitionu}

We say that $X$ has good reduction if it admits a smooth model over $\mathcal{O}_K$. It follows from the smooth and proper base change theorem that if $X$ has good reduction, then the $G_K$-action on $H^2_\et(X_{K^{\rm s}},\Q_\ell)$ is unramified for any prime $\ell\neq \mathrm{char}(k)$. In order to study the extent to which the converse of this is true, we consider the following (strong) notion of semistable reduction.

\begin{AssumptionStar}
  A K3 surface $X$ over $K$ satisfies ($\star$) if
  there exists a finite field extension $L/K$ and a model $\mathcal{X}$ for $X_L$ over $\mathcal{O}_K$ such that:
  \begin{enumerate}
  \item the special fibre of $\mathcal{X}$ is a reduced strict normal crossing divisor;
  \item the relative canonical divisor $\omega_{\mathcal{X}/\mathcal{O}_L}$ is trivial.
  \end{enumerate}
\end{AssumptionStar}

In honour of Kulikov's results \cite{Kul77}, such a model $\mathcal{X}$ is called a \emph{Kulikov} model. 
In equicharacteristic zero, Assumption $(\star)$ is always 
satisfied \cite{Kul77, PP81}. Also, if the residue characteristic $p$ is at least $5$, Assumption $(\star)$ is satisfied whenever $X$ has potential scheme theoretic semistable reduction \cite{LM18, Mau14}. For example, this is satisfied if $X$ admits an ample invertible sheaf $\mathcal{L}$, whose self-intersection number
satisfies $\mathcal{L}^2+4<p$. We refer to \cite[\S3]{LM18} for further results and discussion.

\subsection*{Good reduction in characteristic zero}
If $\mathrm{char}(K)=0$, then converse results, that is, N\'eron-Ogg-Shafarevich criteria for good reductions
of K3 surfaces have been established in \cite{CLL17, LM18, Mat15}. More precisely, the main result (Theorem 1.3) of \cite{LM18} reads as follows.

\begin{theoremn}[Liedtke--Matsumoto] \label{theo: LM18 main} Assume that $\mathrm{char}(K)=0$, let $X$ be a K3 surface over $K$, satisfying Assumption $(\star)$, and let $\ell$ be a prime $\neq \mathrm{char}(k)$. Consider the following three statements:
\begin{enumerate}
\item the $G_K$-action on $H^2_\et(X_{K^{\rm s}},\Q_\ell)$ is unramified; \label{cond: unram} 
\item $X$ has good reduction over a finite and unramified extension $L/K$;
\item $X$ admits a model $\mathcal{X}$ that is projective over $\mathcal{O}_K$ and whose special fibre is a (reduced) K3 surface with at worst RDP singularities.\footnote{For a definition of RDP singularities in arbitrary characteristic, see \cite{Art66}.} 
\end{enumerate}
Then, we have $(1)\Leftrightarrow(2)\Rightarrow (3)$.
\end{theoremn}

\begin{remarku} There exist examples of K3 surfaces over $\Q_p$ with $p\geq5$ that only admit good reduction over a non-trivial unramified extension, see \cite[Theorem 1.6]{LM18}. We will show how to modify these to produce similar examples over Laurent series fields $\F_p\lser{t}$  in \S\ref{sec: counter} below.
\end{remarku}

This begs the question of whether there exists a criterion that sees good reduction `on the nose'. This was the question investigated in \cite{CLL17} and to explain the answer, we note that any K3 surface $X/K$ satisfying $(1)$, $(2)$ and $(3)$ in Theorem \ref{theo: LM18 main} has a `canonical reduction' $Y$, which is a K3 surface over $k$. Concretely, this is given as the minimal resolution of singularities of the special fibre of any RDP model $\mathcal{X}$ as in $(3)$.
We note that $Y$ is unique up to canonical isomorphism and does not depend on the choice of model. The main result (Theorem 1.6) of \cite{CLL17} is then the following.

\begin{theoremn}[Chiarellotto--Lazda--Liedtke] \label{theo: CLL17 main}
Assume that $\mathrm{char}(K)=0$, let $X$ be a K3 surface over $K$, satisfying Assumption $(\star)$ and let $\ell$ be a prime $\neq \mathrm{char}(k)$. Then, then following are equivalent:
\begin{enumerate}
\item $X$ has good reduction over $K$;
\item \label{cond: NOS} the $G_K$-representation $H^2_\et(X_{K^{\rm s}},\Q_\ell)$ is unramified, and there is an isomorphism
\[ H^2_\et(X_{K^{\rm s}},\Q_\ell) \overset{\cong}{\longrightarrow} H^2_\et(Y_{\bar k},\Q_\ell)\]
of $G_k$-representations.
\end{enumerate}
\end{theoremn}

\begin{remarku} When $\mathrm{char}(k)=p>0$, there are also $p$-adic versions of both results in terms of crystalline cohomology, see \cite[Theorem 1.1 and Theorem 1.6]{CLL17}.\end{remarku}

\subsection*{Good reduction in positive equicharacteristic}
The classical $\ell$-adic good reduction criterion of Serre--Tate for abelian varieties \cite{ST68} makes no restriction on the generic characteristic of the base. Neither does its $p$-adic analogue \cite[Exp. IX, Thm. 5.13]{SGA7i}, phrased in terms of the $p$-divisible group associated to an abelian variety. The purpose of this note is therefore to explain how to extend the above results on good reduction of K3 surfaces to the case when $\mathrm{char}(K)=p>0$, both in terms of $\ell$-adic and $p$-adic cohomology. As we already observed in \cite{CLL17}, essentially everything in the $\ell$-adic case from \cite{LM18,CLL17} goes through if Assumption ($\star$) is strengthened to require the given finite extension $L/K$ to be \emph{separable} - this is Assumption ($\star^{\rm s}$) in \S\ref{sec: recap}. The key point, then, is to deal with the problem of descending good reduction over a finite and purely inseparable extension $L/K$. Our main result in this direction is the following.

\begin{theoremu}[Corollary \ref{cor: gr desc}] Let $L/K$ be a finite and purely inseparable extension and $X/K$ a K3 surface. Then, $X$ has good reduction over $K$ if and only if it has good reduction over $L$.
\end{theoremu} 

In fact, the main property of K3 surfaces needed is that they do not admit non-zero global vector fields. 
Thus, we actually prove a more general statement about purely inseparable descent of models (Theorem \ref{thm: gr desc}) that may be of independent interest.
This result already gives the desired extensions of the above results to $\ell$-adic cohomology.

\begin{corollaryu} Theorems \ref{theo: LM18 main} and \ref{theo: CLL17 main} hold without any restriction on $\mathrm{char}(K)$.
\end{corollaryu}

In order to provide $p$-adic versions of these results in equicharacteristic $p>0$, we need to use the Robba ring valued version of rigid cohomology, as was developped in reasonable generality in \cite{LP16}.\footnote{That this cohomology theory can be used to rephrase the `classical' $p$-adic criterion of Grothendieck for abelian varieties is explained in \cite[\S5.3]{LP16}} Let $W=W(k)$ denote the ring of Witt vectors of $k$, write $\kappa=W[1/p]$ for its fraction field, and $\mathcal{R}$ for the Robba ring over $\kappa$. Then for any variety $X/K$, there are rigid cohomology groups $H^i_\rig(X/\mathcal{R})$,
which are $\pn$-modules over $\mathcal{R}$, and behave completely analogously to the $\ell$-adic \'etale cohomology groups for $\ell\neq p$. In fact, since we will only be working with smooth and proper varieties it is possible to phrase everything purely in terms of crystalline cohomology, see \S\ref{sec: robba} for a detailed discussion. In the $p$-adic case in equicharacteristic, the analogue of Theorem \ref{theo: LM18 main}
is the following.

\begin{theoremn} \label{theo: LM18 main equi} Assume that $\mathrm{char}(K)=p>0$, and let $X$ be a K3 surface over $K$, satisfying Assumption $(\star)$. Consider the following statements:
\begin{enumerate}
\item for some prime (resp. all primes) $\ell\neq p$, the $G_K$-action on $H^2_\et(X_{K^{\rm s}},\Q_\ell)$ is unramified;
\item the $\pn$-module $H^2_\rig(X/\mathcal{R})$ admits a basis of horizontal sections;
\item $X$ has good reduction over a finite, unramified extension $L/K$;
\item $X$ admits a model $\mathcal{X}$ that is projective over $\mathcal{O}_K$, and whose special fibre is a K3 surface with at worst RDP singularities. 
\end{enumerate}
Then, we have $(1)\Leftrightarrow(2)\Leftrightarrow (3) \Rightarrow (4)$.
\end{theoremn}

Again, if we let $Y$ be the minimal resolution of singularities of $\mathcal{X}_k$, then this is a K3 surface over $k$, which only depends on $X$ up to canonical isomorphism. We then have the following analogue of Theorem \ref{theo: CLL17 main} in equicharacteristic $p>0$.

\begin{theoremn} \label{theo: CLL17 main equi}
Assume that $\mathrm{char}(K)=p>0$, and let $X$ be a K3 surface over $K$, satisfying Assumption $(\star)$. Then then following are equivalent:
\begin{enumerate}
\item $X$ has good reduction over $K$;
\item for some prime (resp. all primes) $\ell \neq p$, the $G_K$-representation $H^2_\et(X_{K^{\rm s}},\Q_\ell)$ is unramified, and there is an isomorphism
\[ H^2_\et(X_{K^{\rm s}},\Q_\ell) \overset{\cong}{\longrightarrow} H^2_\et(Y_{\bar k},\Q_\ell)\]
of $G_k$-representations.
\item the $\pn$-module $H^2_\rig(X/\mathcal{R})$ admits a basis of horizontal sections, and there is an isomorphism
\[ H^2_\rig(X/\mathcal{R})^{\nabla=0} \overset{\cong}{\longrightarrow} H^2_\mathrm{rig}(Y/\kappa) \]
of $F$-isocrystals over $\kappa$.
\end{enumerate}
\end{theoremn}

Finally, there also exist examples similar to the ones of \cite[Theorem 1.6]{LM18} over Laurent series fields
that show that in general one \emph{cannot} expect to choose the finite and unramified extension $L/K$ from (3) 
of Theorem \ref{theo: LM18 main equi} to be trivial.

\begin{theoremu}[\ref{theo: counter example}]
For every prime $p\geq5$, there exists a smooth K3 surface over $\F_p\lser{t}$ 
  that admits good reduction over $\F_{p^2}\lser{t}$ but not over $\F_p\lser{t}$.
\end{theoremu}

\subsection*{Organisation}
Let us now give a brief summary of the contents of this note. In \S\ref{sec: robba} we will give a brief overview of $p$-adic cohomology over positive characteristic local fields. In \S\ref{sec: recap} we will roughly outline the arguments from \cite{LM18} and \cite{CLL17}, and explain why they carry over to the equicharacteristic $p>0$ case under a suitable strengthening of Assumption $(\star)$ that we call Assumption ($\star^{\rm s}$). In \S\ref{sec: descent} we will discuss descent under purely inseparable fpqc covers, and prove our main result that smooth models of K3 surfaces descend under purely inseparable field extensions. In \S\ref{sec: together} we will bring everything together to prove our main results, Theorems \ref{theo: LM18 main equi} and \ref{theo: CLL17 main equi}. Finally, in \S\ref{sec: counter} we will construct counter-examples asserted by Theorem \ref{theo: counter example}.

\subsection*{Acknowledgements} We would like to thank Yuya Matsumoto for explaining to us why arithmetic threefold flops still exist in equicharacteristic $2$, as well as the two anonymous referees for a careful reading of the paper. The second named author is support by the Netherlands Organisation for Scientific Research (NWO). The third named author is supported by the ERC Consolidator Grant 681838 ``K3CRYSTAL''.

\section{Review of \texorpdfstring{$p$}{p}-adic cohomology in equicharacteristic \texorpdfstring{$p>0$}{pg0}}\label{sec: robba}

In this section we will briefly review the facts needed on $p$-adic cohomology when $\mathrm{char}(K)=p>0$. 
Thus, $\mathcal{O}_K$ is an excellent and Henselian DVR of characteristic $p>0$ with fraction field $K$ and perfect residue field $k$. We let $W=W(k)$ denote the ring of Witt vectors of $k$, $\kappa$ its fraction field, and $\sigma$ the $p$-power Frobenius on $W$ and $\kappa$.

\begin{remark} The choice of notation $\kappa$ for the fraction field of $W(k)$ is distinctly ``non-standard''. However, the more usual choices of $K_0$ or $K$ are unfortunately excluded by the use of $K$ for our given equicharacteristic local field.
\end{remark}

In this situation, a choice of uniformiser $\varpi\in\mathcal{O}_K$ induces an isomorphism $\widehat{K}\cong k\lser{\varpi}$ between the completion of $K$ and a Laurent series field over $k$. We will denote by $\mathcal{R}$ the Robba ring over $\kappa$, that is, the ring consisting of those Laurent series $\sum_i a_it^i$ with $a_i\in\kappa$
such that:
\begin{itemize} \item for all $\rho<1$, $\norm{a_i}\rho^i\rightarrow 0$ as $i\rightarrow \infty$;
\item for some $\eta<1$, $\norm{a_i}\eta^i\rightarrow 0$ as $i\rightarrow -\infty$.
\end{itemize}
In other words, $\mathcal{R}$ is the ring of functions convergent on some semi-open annulus $\eta\leq \norm{t}<1$. The ring of integral elements $\mathcal{R}^{\mathrm{int}}$ (that is, those with $a_i\in W$) is therefore a lift of $\widehat{K}$ to characteristic $0$, in the sense that mapping $t\mapsto \varpi$ induces an isomorphism $\mathcal{R}^{\mathrm{int}}/(p)\cong \widehat{K}$. We will denote by $\sigma$ a Frobenius on $\mathcal{R}$, that is, a continuous $\sigma$-linear endomorphism preserving $\mathcal{R}^{\mathrm{int}}$ and lifting the absolute $p$-power Frobenius on $\widehat{K}$.
We will moreover assume that $\sigma(t)=ut^p$ for some $u\in (W\pow{t}\otimes_W \kappa)^\times$. The reader is welcome to assume that $\sigma(\sum_i a_it^i)=\sum_i \sigma(a_i)t^{ip}$. Let $\partial_t:\mathcal{R}\rightarrow \mathcal{R}$ denote the derivation given by differentiation with respect to $t$.

\begin{definition} \label{pnm1} A $\pn$-module over $\mathcal{R}$ is a finite free $\mathcal{R}$-module $M$ together with:
\begin{itemize} \item a connection, that is, a $\kappa$-linear map $\nabla:M\rightarrow M$ such that
\[ \nabla(rm)=\partial_t(r)m+r\nabla(m)\;\;\;\;\text{for all}\;\; r\in \mathcal{R} \;\; \text{and} \;\;m\in M;\]
\item a horizontal Frobenius $\varphi:\sigma^*M:=M\otimes_{\mathcal{R},\sigma} \mathcal{R} \isomto M$.
\end{itemize}
\end{definition}

Then, $\pn$-modules over $\mathcal{R}$ should be considered as $p$-adic analogues of Galois representations. For example, they satisfy a local monodromy theorem, see \cite{Ked04a}. More specifically, the connection $\nabla$ should be viewed as an analogue of the action of the inertia subgroup $I_K$ and the Frobenius $\varphi$ the action of some Frobenius lift in $G_K$. The analogue for $\pn$-modules of a Galois representation being unramified is therefore the connection acting trivially, or, in other words, the $\pn$-module admitting a basis of horizontal sections.

The $p$-adic completion $\widehat{\mathcal{R}}^\mathrm{int}$ of the integral Robba ring is a Cohen ring for $\widehat{K}$ and hence, any smooth and proper $\widehat{K}$-variety $Y$ has crystalline cohomology groups
\[ H^i_\mathrm{cris}(Y/\widehat{\mathcal{R}}^\mathrm{int}), \]
which are $\pn$-modules over $\widehat{\mathcal{R}}^\mathrm{int}$, see for example \cite[\S.3.1]{GM87}. It follows from \cite[Theorem 7.0.1]{Ked00} (where the notations $\Gamma= \widehat{\mathcal{R}}^\mathrm{int}$ and $\Gamma^\dagger=\mathcal{R}^\mathrm{int}$ are used) that the crystalline cohomology groups of any smooth and proper variety $Y/\widehat{K}$ descend uniquely to $\pn$-modules
\[ H^i_\mathrm{cris}(Y/\mathcal{R}^\mathrm{int}) \]
over $\mathcal{R}^\mathrm{int}$. Therefore, we may define for any smooth and proper variety $X/K$
\[ H^i_\rig(X/\mathcal{R}) := H^i_\mathrm{cris}(X_{\widehat{K}}/\mathcal{R}^\mathrm{int})\otimes_{\mathcal{R}^\mathrm{int}} \mathcal{R} \]
as $\pn$-modules over $\mathcal{R}$. We have the following analogue of the smooth and proper base change theorem.

\begin{proposition} \label{prop: spbc} Let $\mathcal{X}\rightarrow \spec{\mathcal{O}_K}$ be a smooth and proper morphism of algebraic spaces, whose generic fibre is a scheme. Then, for any $i\geq 0$ there exists a canonical isomorphism
\[ H^i_\rig(\mathcal{X}_K/\mathcal{R})^{\nabla=0} \overset{\cong}{\longrightarrow} H^i_\mathrm{rig}(\mathcal{X}_k/\kappa), \]
of $F$-isocrystals over $K$.
\end{proposition}

\begin{proof}
Since $\mathcal{X}_k$ is smooth over $k$, the right hand side is isomorphic to the rational log-crystalline cohomology of $\mathcal{X}_k$, equipped with the log structure
\[ \N \rightarrow \mathcal{O}_{\mathcal{X}_k} ,\;\;\;\;1\mapsto 0 \]
so we may apply \cite[Proposition 2.3]{CL16}.
\end{proof}

We will also need cycle class maps in $p$-adic cohomology. For any smooth and proper variety $Y/\widehat{K}$, homomorphisms
\[ [-]:\mathrm{CH}^d(Y) \rightarrow H^{2d}_\mathrm{cris}(Y/\widehat{\mathcal{R}}^\mathrm{int})_\Q \]
are constructed in \cite[Theorem 4.3.1]{GM87}, and since their image lands in the subspace of horizontal sections on which Frobenius acts as $p^d$, it follows from Kedlaya's full faithfullness theorem \cite[Theorem 5.1]{Ked04c} that for any smooth and proper variety $Y/\widehat{K}$, the codimension-$d$ crystalline cycle class map for $Y$ actually takes values in $H^{2d}_\mathrm{cris}(Y/\mathcal{R}^\mathrm{int})_{\Q}^{\nabla=0,\varphi=p^d}$. This allows us to define, for any smooth and proper variety $X/K$, cycle class maps
\[ [-]:\mathrm{CH}^d(X) \rightarrow H^{2d}_\rig(X/\mathcal{R}) \]
landing in the subspace of horizontal sections on which Frobenius acts as $p^d$. On the other hand, for any smooth and proper variety $Y/k$, we also have crystalline cycle class map
\[ [-]:\mathrm{CH}^d(Y) \rightarrow H^{2d}_\rig(Y/\kappa). \]
The following is the $p$-adic analogue of \cite[Lemma 5.6]{LM18}, see also \cite[\S1]{CCM13}.

\begin{proposition} \label{prop: cyccomp} Let $\mathcal{X}\rightarrow \spec{\mathcal{O}_K}$ be a smooth and proper morphism of algebraic spaces, whose special and generic fibres are schemes. Then, for any closed subspace $\mathcal{Z}\hookrightarrow \mathcal{X}$ of constant codimension $d$ and flat over $\mathcal{O}_K$, the isomorphism
\[  H^{2d}_\rig(\mathcal{X}_K/\mathcal{R})^{\nabla=0} \overset{\cong}{\longrightarrow} H^{2d}_\mathrm{rig}(\mathcal{X}_k/\kappa) \]
sends $[\mathcal{Z}_K]$ to $[\mathcal{Z}_k]$.
\end{proposition}

\begin{proof} We may assume that $K=\widehat{K}$ is complete. By combining Chow's lemma with de\thinspace Jong's theorem on alterations \cite{dJ96} we may choose a finite extension $L/K$, a projective and strictly semistable scheme $\mathcal{Y}\rightarrow \spec{\mathcal{O}_L}$ and an alteration $\pi: \mathcal{Y}\rightarrow \mathcal{X}$ of algebraic spaces over $\mathcal{O}_K$. 

Let $k_L$ denote the residue field of $L$, $W_L=W(k_L)$ and fix a lift of $\mathcal{O}_K\rightarrow \mathcal{O}_L$ of the form $W\pow{t}\rightarrow W_L\pow{t_L}$. Let $\mathcal{R}_L$ denote a copy of the Robba ring over $\kappa_L=W_L[1/p]$ with parameter $t_L$, and $\mathcal{R}\rightarrow \mathcal{R}_L$ the induced finite flat extension. This sends $\mathcal{R}^\mathrm{int}$ into $\mathcal{R}_L^\mathrm{int}$. Let $\mathcal{Y}^\times$ denote the scheme $\mathcal{Y}$ equipped with the fs log structure induced by the special fibre, and $\mathcal{Y}_{k_L}^\times$ the special fibre of this log scheme. Let $W_{k_L}^\times$ denote the ring $W_{k_L}$ equipped with the log structure \begin{align*}
\N &\rightarrow W_{k_L} \\
1 &\mapsto 0
\end{align*}
and $W_L\pow{t_L}^\times$ for $W_L\pow{t_L}$ equipped with the log structure
\begin{align*}
\N &\rightarrow W_{k_L}\pow{t_L} \\
1 &\mapsto t_L.
\end{align*}
By \cite[Theorem 5.46]{LP16} there is, for all $i\geq0$, an isomorphism
\[ H^i_\rig(\mathcal{Y}_L/\mathcal{R}_L)^{\nabla=0} \overset{\cong}{\longrightarrow} H^i_{\log\text{-}\cris}(\mathcal{Y}_{k_L}^\times/W_L^\times)\otimes \Q, \]
fitting into a commutative diagram
\[ \xymatrix{ H^i_\rig(\mathcal{Y}_L/\mathcal{R}_L)^{\nabla=0} \ar[r]^-{\cong} &  H^i_{\log\text{-}\cris}(\mathcal{Y}_{k_L}^\times/W^\times_L)\otimes \Q \\ 
H^i_\rig(\mathcal{X}_K/\mathcal{R})^{\nabla=0} \ar[u] \ar[r]^-{\cong} & H^i_{\rig}(\mathcal{X}_k/\kappa). \ar[u]  } \]
Both vertical maps are injective by Poincar\'e duality (in the log crystalline case, see \cite{Tsu99a}). In particular, it suffices to verify that $[\mathcal{Z}_K]$ and $[\mathcal{Z}_k]$ map to the same element in $H^{2d}_{\log\text{-}\cris}(\mathcal{Y}_{k_L}^\times/W_L^\times)\otimes \Q$.

For any algebraic space over $\mathcal{O}_K$, we will let $K(-)$ denote the $K$-theory of vector bundles. We therefore have total Chern class maps 
\begin{align*} c: K_0(\mathcal{X}_K) &\rightarrow \bigoplus_{n\geq0} H^{2n}_\cris(\mathcal{X}_K/\widehat{\mathcal{R}}^\mathrm{int}) \\ 
c: K_0(\mathcal{Y}_K) &\rightarrow \bigoplus_{n\geq0} H^{2n}_\cris(\mathcal{Y}_K/\widehat{\mathcal{R}}_L^\mathrm{int}) \\
c: K_0(\mathcal{X}_k) &\rightarrow \bigoplus_{n\geq0} H^{2n}_\cris(\mathcal{X}_k/W)
\end{align*}
taking values in crystalline cohomology, as constructed in \cite{BI70}. Using exactly the same method as in \cite{BI70} we can also construct total Chern class maps 
\begin{align*} c: K_0(\mathcal{Y}) &\rightarrow \bigoplus_{n\geq0} H^{2n}_{\log\text{-}\cris}(\mathcal{Y}^\times/W_L\pow{t_L}^\times) \\ 
c: K_0(\mathcal{Y}_{k_L}) &\rightarrow \bigoplus_{n\geq0} H^{2n}_{\log\text{-}\cris}(\mathcal{Y}^\times_{k_L}/W_L^\times)
\end{align*}
in log-crystalline cohomology. Indeed, if we have a $p$-adic log PD base $S$, (e.g. $W^\times$ or $W\pow{t}^\times$) and an fs log scheme $X$ equipped with a log smooth and proper morphism $X\rightarrow S$, then for any rank $r$ vector bundle $\mathcal{E}$ on $X$, we can endow its projectivisation $g:\P(\mathcal{E})\rightarrow X$ with the pullback log-structure from $X$. The first Chern class
\[ \xi:= c_1(\mathcal{O}_{\P(\mathcal{E})}(1))
\in H^2_{\log\text{-}\cris}(\P(\mathcal{E})/W)\]
can be defined directly, and in order to follow the contruction in \cite{BI70} we need to know that the cohomology $H^*_{\log\text{-}\cris}(\P(\mathcal{E})/W)$ decomposes as a direct sum
\[ H^*_{\log\text{-}\cris}(\P(\mathcal{E})/S) \cong \bigoplus_{i=0}^{r-1}H^*_{\log\text{-}\cris}(X/S)\cdot \xi^i.	 \]
We can turn this into a local statement, namely that the higher direct image $\mathbf{R}g_*\mathcal{O}^{\log\text{-}\cris}_{\P(\mathcal{E})/S}$ decomposes as
\[ \mathbf{R}g_*\mathcal{O}^{\log\text{-}\cris}_{\P(\mathcal{E})/S} \cong \bigoplus_{i=0}^{r-1}\mathcal{O}^{\log\text{-}\cris}_{X/S} \cdot \xi^i,  \]
where we have abused notation and also written $\xi\in H^0_{\log\text{-}\cris}(X/S,\mathbf{R}^2g_* \mathcal{O}^{\log\text{-}\cris}_{\P(\mathcal{E})/S} )$ for the image of the first Chern class of $\mathcal{O}_{\P(\mathcal{E})}(1)$ under the edge morphism of the Leray spectral sequence. Since this statement is now local, we can assume that in fact $\mathcal{E}=\mathcal{O}_{X}^{\oplus r}$, thus $\P(\mathcal{E})=\P^{r-1}_X$. In this case the claim follows from the smooth and proper base change theorem in log-crystalline cohomology \cite[Theorem 6.10]{Kat89}, together with the standard computation of the crystalline cohomology of projective space.

Now, since $\mathcal{Z}$ is flat, the fibres $\mathcal{Z}_K$ and $\mathcal{Z}_k$ are also of codimension $d$ in $\mathcal{X}_K$ and $\mathcal{X}_k$ respectively. Since $\mathcal{X}_K$ and $\mathcal{X}_k$ are both regular and  projective schemes, the $K$-theory of vector bundles coincides with that of coherent sheaves, so we have well-defined classes $[\mathcal{O}_{\mathcal{Z}_K}]\in K_0(\mathcal{X}_K)$ and $[\mathcal{O}_{\mathcal{Z}_k}]\in K_0(\mathcal{X}_k)$. Essentially by definition \cite[Theorem 4.3.1]{GM87}, we have
\[ [\mathcal{Z}_K]=\frac{(-1)^{d-1}}{(d-1)!}c_d([\mathcal{O}_{\mathcal{Z}_K}])  \mbox{ \quad and \quad } [\mathcal{Z}_k] =\frac{(-1)^{d-1}}{(d-1)!}c_d([\mathcal{O}_{\mathcal{Z}_k}]). \]
Since $\mathcal{O}_{\mathcal{Z}_K}$ and $\mathcal{O}_{Z_k}$ are strictly perfect complexes on $\mathcal{X}_K$ and $\mathcal{X}_k$ respectively, we get well-defined classes $[\mathbf{L}\pi^*\mathcal{O}_{\mathcal{Z}_K}]\in K_0(\mathcal{Y}_{L})$ and $[\mathbf{L}\pi^*\mathcal{O}_{\mathcal{Z}_k}]\in K_0(\mathcal{Y}_{k_L})$. By functoriality of Chern classes, it suffices to show that the isomorphism
\[ H^{2d}_\rig(\mathcal{Y}_L/\mathcal{R}_L)^{\nabla=0} \overset{\cong}{\longrightarrow} H^{2d}_{\log\text{-}\cris}(\mathcal{Y}_{k_L}^\times/W_L^\times)\otimes \Q \]
sends $c_d([\mathbf{L}\pi^*\mathcal{O}_{\mathcal{Z}_K}])$ to $c_d([\mathbf{L}\pi^*\mathcal{O}_{\mathcal{Z}_k}])$. Since $\mathcal{X}$ is regular we know that $\mathcal{O}_\mathcal{Z}$ is a perfect complex of $\mathcal{O}_\mathcal{X}$-modules, and since $\mathcal{Y}$ is a projective scheme, $\mathbf{L}\pi^*\mathcal{O}_\mathcal{Z}$ is actually a \emph{strictly} perfect complex on $\mathcal{Y}$. It therefore has a well defined class in $K_0(\mathcal{Y})$, which restricts to the class of $\mathbf{L}\pi^*\mathcal{O}_{\mathcal{Z}_K}$ (resp. $\mathbf{L}\pi^*\mathcal{O}_{\mathcal{Z}_k}$) on the generic (resp. special) fibre. Given the construction of the isomorphism
\[ H^{2d}_\rig(\mathcal{Y}_L/\mathcal{R}_L)^{\nabla=0} \overset{\cong}{\longrightarrow} H^{2d}_{\log\text{-}\cris}(\mathcal{Y}_{k_L}^\times/W_L^\times)\otimes \Q \]
it now simply suffices to note that the diagram
\[ \xymatrix{ K_0(\mathcal{Y}_K) \ar[d]^c & K_0(\mathcal{Y})\ar[d]^c \ar[l]  \ar[r] & K_0(\mathcal{Y}_k)\ar[d]^c \\ \bigoplus_{n\geq0} H^{2n}_{\cris}(\mathcal{Y}_L/\widehat{\mathcal{R}}_L^\mathrm{int}) &
\bigoplus_{n\geq0} H^{2n}_{\log\text{-}\cris}(\mathcal{Y}^\times/ W_L\pow{t_L}^\times) \ar[r]\ar[l] & \bigoplus_{n\geq0} H^{2n}_{\log\text{-}\cris}(\mathcal{Y}_{k_L}^\times/ W_L^\times)  } \]
commutes.
\end{proof}

\begin{remark} We need to argue on the alteration $\mathcal{Y}$ rather than on the model $\mathcal{X}$ itself because of the potential difference between the $K$-group of vector bundles (equivalently: strictly perfect complexes) and that of perfect complexes. They coincide on regular schemes (like $\mathcal{Y}$), but it is not known whether the same is true for algebraic spaces (like $\mathcal{X}$). If it were, the above proof could be significantly simplified.
\end{remark}

\section{Descending good reduction under separable extensions} \label{sec: recap}

In this section, we will explain how and why the main results of \cite{LM18} and \cite{CLL17} carry over in equicharacteristic $p$ under a strengthening of Assumption ($\star$), which we will call Assumption ($\star^{\rm s}$). Under this stronger assumption, we also establish a version of these results in terms of the Robba ring valued $p$-adic cohomology discussed in \S\ref{sec: robba}.

\begin{AssumptionStarSep}  
A K3 surface $X/K$ satisfies ($\star^{\rm s}$) if it satisfies Assumption ($\star$) and we can moreover take the finite extension $L$ to be separable over $K$.
\end{AssumptionStarSep}

Then we can use the arguments of \cite{LM18} word-for-word to prove the following.

\begin{theorem} \label{theo: LM18 main sep} Let $X$ be a K3 surface over $K$, satisfying Assumption $(\star^{\rm s})$, and let $\ell$ be a prime $\neq p$. Consider the following three statements:
\begin{enumerate}
\item \label{num: urQl} the $G_K$-action on $H^2_\et(X_{K^{\rm s}},\Q_\ell)$ is unramified;
\item \label{num: grux} $X$ has good reduction over a finite, unramified extension $L/K$;
\item \label{num: RDPmod} $X$ admits a model $\mathcal{X}$ that is projective over $\mathcal{O}_K$, and whose special fibre is a K3 surface with at worst RDP singularities. 
\end{enumerate}
Then, we have $(\ref{num: urQl})\Leftrightarrow(\ref{num: grux})\Rightarrow (\ref{num: RDPmod})$.
\end{theorem}

\begin{proof}
 The point is that the characteristic $0$ hypothesis is only used in \cite{LM18} in order to ensure that $X$ admits a Kulikov model over a Galois extension $L/K$, 
which enables arguments via Galois descent to be used. 
Thus, if we admit this from the start, the proofs in \cite{LM18} go through unchanged, with one important exception. 

This exception occurs in the proof of \cite[Proposition 4.2]{LM18}, showing the existence of flops on relative surfaces over $\mathcal{O}_K$ with numerically trivial relative canonical divisor. One step of the construction uses the fact that a deformation
\[ \spf{\frac{\mathcal{O}_K\pow{x,y,z}}{z^2-H_1(x,y)z-H_0(x,y)}}\]
of a rational double point singularity over $k$ admits a non-trivial involution
\[ z\mapsto H_1(x,y)-z.\]
This will clearly remain true provided $\mathrm{char}(K)\neq 2$, and to show that it also remains true when $\mathrm{char}(K)=2$, we need to explain why we will always have $H_1(x,y)\neq 0$. 

This was pointed out to us by Yuya Matsumoto. Indeed, suppose that we have $H_1(x,y) = 0$. Then, since the equation
\[ \frac{\partial H_0}{\partial x} = \frac{\partial H_0}{\partial y} =
0\]
has the solution $(x,y) = (0,0)$ on the special fibre, after possibly replacing $K$ by a finite extension, it also has a solution $(x,y) = (\xi, \eta)$ on the generic fibre. This would mean that the generic fibre has a singular point $(x,y,z) = (\xi, \eta, \sqrt{H_0(\xi,\eta)})$. Therefore smoothness of the generic fibre means that we must have $H_1(x,y)\neq 0$ after all.
\end{proof}

In order to have a version of this result also in terms of the Robba ring valued $p$-adic cohomology discussed in \S\ref{sec: robba}, the first step is to have an analogue of \cite[Theorem 1.1]{Mat15} in this cohomology theory. This is provided by \cite[Theorem 6.4]{CL16}. The next key point will be to show that smooth models descend under finite, \emph{totally ramified} extensions, under suitable cohomological assumptions. More precisely, we need to establish the analogue of \cite[Proposition 5.8]{LM18} in Robba-ring value rigid cohomology. 

Thus, let $X$ be a K3 surface over $K$. Suppose that there exists a totally ramified Galois extension $L/K$ with Galois group $G$ such that $X_L$ 
admits a smooth model $\mathcal{X}\rightarrow \spec{\mathcal{O}_L}$. Let $\mathcal{R}$ be the Robba ring over $\kappa$ and let $\mathcal{R}_L$ denote the unique unramified extension of $\mathcal{R}$ corresponding to $L/K$ as in \cite{Mat95}, see also the proof of Proposition \ref{prop: cyccomp}. Then, there is a natural $G$-action on the $\pn$-module $H^2_\rig(X_L/\mathcal{R}_L)$ and hence, on its subspace $H^2_\rig(X_L/\mathcal{R}_L)^{\nabla=0}$ 
of horizontal sections.

\begin{proposition} \label{prop: extension}
Under the previous assumptions, assume that the $G$-action on $H^2_\rig(X_L/\mathcal{R}_L)^{\nabla=0}$ is trivial.
 Then, the $G$-action on $X_L$ extends to $\mathcal{X}$ and the induced $G$-action on the special fibre $\mathcal{X}_{k}$ is trivial.
\end{proposition}

\begin{proof}
The key ingredient here is Proposition \ref{prop: cyccomp}. Indeed, given this we can copy the proof of \cite[Proposition 5.5]{LM18} word for word to show that the $G$-action on $X_L$ extends to $\mathcal{X}$. To show that the induced action on the special fibre is trivial, we use \cite[Corollary 2.5]{Ogu79} and \cite[Theorem 1.4]{Keu16}, which give injectivity of the natural map
\[\mathrm{Aut}(Y) \rightarrow \mathrm{GL}(H^2_\rig(Y/\kappa))  \]
for K3 surfaces over $k$.
\end{proof}

After this preparation, we arrive at the following $p$-adic version of Theorem \ref{theo: LM18 main sep}.

\begin{theorem} \label{theo: LM18 main padic} Assume that $\mathrm{char}(K)=p>0$, let $X$ be a K3 surface over $K$, satisfying  Assumption $(\star^{\rm s})$, and let $\ell$ be a prime $\neq p$. Then $(\ref{num: urQl})$ and $(\ref{num: grux})$ in Theorem \ref{theo: LM18 main sep} are both equivalent to the following:
\begin{enumerate}\item[$(1)'$] the $\pn$-module $H^2_\rig(X/\mathcal{R})$ admits a basis of horizontal sections. 
\end{enumerate}
\end{theorem}

\begin{proof}
The implication $(\ref{num: grux})\Rightarrow (1)'$ follows from Proposition \ref{prop: spbc}, together with the fact that a $\pn$-module $M$ over $\mathcal{R}$ admits a basis of horizontal sections if and only if it does so after replacing $\kappa$ by a finite extension.

For the implication $(1)'\Rightarrow (\ref{num: grux})$, the main extra input we need is the correct analogue of \cite[Proposition 5.8]{LM18}, which is provided by 
Proposition \ref{prop: extension}. The completion of the proof of Theorem \ref{theo: LM18 main padic} is now geometric and identical to the $\ell$-adic case.
\end{proof} 

We can also extend the results of \cite{CLL17} to equicharacteristic with essentially no difficulties. 

\begin{theorem} 
Assume that $\mathrm{char}(K)=p>0$ and let $X$ be a K3 surface over $K$, satisfying  Assumption $(\star^{\rm s})$. Then, then following are equivalent:
\begin{enumerate}
\item $X$ has good reduction over $K$;
\item for some prime (resp. all primes) $\ell \neq p$, the $G_K$-representation $H^2_\et(X_{K^{\rm s}},\Q_\ell)$ is unramified, and there is an isomorphism
\[ H^2_\et(X_{K^{\rm s}},\Q_\ell) \overset{\cong}{\longrightarrow} H^2_\et(Y_{\bar k},\Q_\ell)\]
of $G_k$-representations.
\item the $\pn$-module $H^2_\rig(X/\mathcal{R})$ admits a basis of horizontal sections, and there is an isomorphism
\[ H^2_\rig(X/\mathcal{R})^{\nabla=0} \overset{\cong}{\longrightarrow} H^2_\mathrm{rig}(Y/\kappa) \]
of $F$-isocrystals over $\kappa$.
\end{enumerate}
\end{theorem}

\begin{proof}
The existence of flops that is required in \cite[Proposition 7.5]{CLL17}, follows from the same argument we used during the proof of Theorem \ref{theo: LM18 main sep}. Apart from this,
everything in \cite[\S7]{CLL17} works exactly the same in positive equicharacteristic.
The same is true for \cite[\S8]{CLL17}, replacing $\mathbb{D}_\mathrm{cris}(H^2_\et(X_{\overline{K}},\Q_p))$ everywhere by $H^2_\rig(X/\mathcal{R})^{\nabla=0}$, and using Proposition \ref{prop: cyccomp} in order to prove the correct analogue of \cite[Lemma 8.1]{CLL17}. Everything in \cite[\S9]{CLL17} goes through unchanged.
\end{proof}

\section{Descending good reduction under purely inseparable extensions} \label{sec: descent}

In this section, we shall show that smooth models of K3 surfaces descend under purely inseparable extensions of $K$, 
although we will start off in rather more generality than that. 
Let $S$ be a scheme and let $p:S'\rightarrow S$ be a purely inseparable (that is, radicial, or universally injective) fpqc covering.

\begin{lemma} The map $p:S'\rightarrow S$ is separated and the diagonal embedding $\Delta: S'\rightarrow S'\times_S S'$ is a nilpotent closed immersion.
\end{lemma}

\begin{proof}
Since $p$ is surjective and universally injective, it is universally bijective. 
In particular, we see that either projection $p_i:S'\times_S S'\rightarrow S'$ is bijective. 
Since $\Delta$ is a section of $p_i$, we deduce that $\Delta$ and $p_i$ are in fact both homeomorphisms. 
As $\Delta$ is a locally closed embedding with closed image, it must therefore be a closed immersion, and $p$ is therefore separated. 
Since $\Delta$ is a closed immersion inducing a homeomorphism on the underlying topological spaces, it is defined by a nilpotent sheaf of ideals.
\end{proof}

We denote by $\mathcal{I}$ the ideal of the diagonal $S'\hookrightarrow S'\times_S S'$, and write
\[ S'^{(n)}:= \underline{\mathrm{Spec}}_{S'\times_S S'}(\mathcal{O}_{S'\times_S S'}/\mathcal{I}^{n+1})\]
for the $n$th infinitesimal neighbourhood of $S'$ in $S'\times_S S'$. 
Let $f:Y\rightarrow S'$ be an algebraic space with relative tangent sheaf $\mathcal{T}_{Y/S'}:=\mathcal{H}\mathrm{om}_{\mathcal{O}_Y}(\Omega^1_{Y/S'},\mathcal{O}_Y)\in \mathrm{QCoh}(Y)$. 
For any morphism of schemes $U\rightarrow S$ we write $(-)_U$ for the fibre product $(-)\times_S U$. 

\begin{proposition} \label{prop: extending descent}Assume that:
\begin{enumerate}
\item the ideal $\mathcal{I}$ is of finite type;
\item each $S'^{(n)}$ is flat over $S'$ (under either projection);
\item $Y$ is smooth and separated over $S'$;
\item $f_*\mathcal{T}_{Y/S'}=0$, and $\mathbf{R}^1f_*\mathcal{T}_{Y/S'}$ is torsion free.
\end{enumerate} 
Then, for any dense open subscheme $U\subset S$, the function
\[ \left\{\text{descent data on }Y\text{ along }S'\rightarrow S \right\} \rightarrow \left\{\text{descent data on }Y_U\text{ along }S_U'\rightarrow U \right\}\]
obtained by restriction is a bijection.
\end{proposition}

\begin{remark} Note that we do not assume $Y$ proper over $S'$. 
However, the condition that $f_*\mathcal{T}_{Y/S'}=0$ will generally force this in practice.
\end{remark}

\begin{proof} Write $p_i^*$ for pullback along either projection $p_i:S'\times_S S'\rightarrow S'$. 
The given function is injective since $Y$ is separated over $S'$, thus an isomorphism $p_1^*Y_U \rightarrow p_2^*Y_U$ can extend 
to at most one isomorphism $p_1^*Y\isomto p_2^*Y$. 
To see that it is surjective, it suffices to show that we can always extend such an isomorphism
(the fact that such an extension will define a descent datum again 
follows from separatedness of $Y\rightarrow S'$ and the corresponding fact for $Y_U$). Since $S'\rightarrow S'\times_S S'$ is a homeomorphism, $S'\times_S S'$ is covered by open sets of the form $V\times_S V$ with $V\subset S'$ open and affine. Therefore appealing once more to the fact that $U\subset S$ is dense and $Y\rightarrow S'$ is separated, we can see that the problem of extending $\alpha_U$ is local on both $S$ and $S'$, which we may therefore assume to be affine, say $S=\spec{R}$ and $S'=\spec{R'}$. (We may have lost surjectivity of $S'\rightarrow S$ but this doesn't matter). We may also assume that $U=\spec{R[g^{-1}]}$ is a basic open affine in $S$.

We will prove that any such $\alpha_U:p_1^*Y_U \rightarrow p_2^*Y_U$ extends by using deformation theory. Let $R'':=R'\otimes_R R'$ and let $I=\ker \left( R''\rightarrow R'\right)$. 
Thus, $I$ is a nilpotent ideal. By assumption, it is finitely generated and hence, $I^N=0$ for $N$ large enough. 
If we write $R'^{(n)}:= R''/I^{n+1}$, then by assumption each $R'^{(n)}$ is flat over $R$, hence so is the kernel $I^n/I^{n+1}$ 
of the surjection $R'^{(n)}\rightarrow R'^{(n-1)}$. 
Let $p_i^{(n)*}Y$ denote the base change to $R'^{(n)}$ along either `projection' map $p_i^{(n)}:R'\rightarrow R^{'(n)}$.
We will show by induction on $n\geq 1$ that $\alpha_U^{(n)}: p_1^{(n)*}Y_U \isomto p_2^{(n)*}Y_U$ extends to an isomorphism 
$p_1^{(n)*}Y \isomto p_2^{(n)*}Y$.
Since $I^N=0$ for large enough $N$, this suffices to prove the result.

The case $n=0$ is trivial and we may therefore suppose that $n\geq 1$ and that we are given $p_1^{(n-1)*}Y \isomto p_2^{(n-1)*}Y$ extending $\alpha_U^{(n-1)}$. 
We can thus view both $p_1^{(n)*}Y$ and $p_2^{(n)*}Y$ as deformations of $p_1^{(n-1)*}Y$ along $R'^{(n)}\twoheadrightarrow R'^{(n-1)}$. 
By flatness of $I^n/I^{n+1}$ over $R'$, the isomorphism classes of such deformations are controlled by
\[ H^1( Y, \mathcal{T}_{Y/R'})\otimes_{R'} I^n/I^{n+1} ,\]
and we are given that these two classes map to the same element in
\[ H^1( Y, \mathcal{T}_{Y/R'})\otimes_{R'} I^n/I^{n+1} \otimes_R R[g^{-1}]. \]
But by the torsion-free hypothesis on $\mathbf{R}^1f_*\mathcal{T}_{Y/S'}$, we know that 
\[ H^1( Y, \mathcal{T}_{Y/R'})\otimes_{R'} I^n/I^{n+1} \hookrightarrow H^1( Y, \mathcal{T}_{Y/R'})\otimes_{R'} I^n/I^{n+1} \otimes_R R[g^{-1}] \]
is injective, from which we deduce that $p_1^{(n)*}Y$ and $p_2^{(n)*}Y$ are isomorphic as deformations of $p_1^{(n-1)*}Y$. 
Moreover, the automorphism group of a given deformation of $p_1^{*(n-1)}Y_U$ is
\[ H^0( Y, \mathcal{T}_{Y/R'})\otimes_{R'} I^n/I^{n+1} \otimes_R R[g^{-1}] =0, \]
since $f_*\mathcal{T}_{Y/S'}=0$, from which we deduce that \emph{any} isomorphism $p_1^{(n)*}Y \isomto p_2^{*(n)}Y$ as deformations of $p_1^{(n-1)*}Y$ 
restricts to the given isomorphism $\alpha_U^{(n)}: p_1^{(n)*}Y_U \isomto p_2^{(n)*}Y_U$. \end{proof}

Concerning effectivity of descent, we have the following, presumably well-known, result.

\begin{proposition} \label{prop: effective} 
 For any algebraic space $Y\rightarrow S'$, all descent data on $Y$ along $p:S'\rightarrow S$ are effective.
\end{proposition}

\begin{proof}
This is entirely straightforward for schemes and not much harder for algebraic spaces. 
Let $Y\cong [R \rightrightarrows U]$ be a presentation of $Y$ via an \'etale equivalence relation of affine schemes over $S'$. 
Let $\alpha:p_1^*Y \isomto p_2^*Y$ be an isomorphism defining a descent datum on $Y$. 
Since the diagonal $S'\rightarrow S'\times_S S'$ is a nilpotent closed immersion, it is a universal homeomorphism, and we can deduce from topological invariance 
of the small \'etale site (in particular, from the equivalence $\mathrm{\acute{E}t}(Y)\cong \mathrm{\acute{E}t}(p_i^*Y)$ for each $i$) 
that we can uniquely extend $\alpha$ to an isomorphism
\[ p_1^*R\rightrightarrows p_1^*U \cong p_2^* R\rightrightarrows p_2^*U \] 
of \'etale equivalence relations. This isomorphism is compatible on $S'\times_S S'\times_S S'$ and hence (since $U$ and $R$ are affine), descends to an \'etale equivalence 
relation $R_S \rightrightarrows U_S $ of affine schemes over $S$. 
Thus, taking $X\cong [R_S\rightrightarrows U_S]$ to be the corresponding quotient we get the required descent of $Y$ to $S$.
\end{proof}

We will apply these results as follows. 
Let $\mathcal{O}_K$ be our fixed excellent and Henselian DVR with fraction field $K$. 
Let $X$ be a smooth, proper, and geometrically connected scheme over $K$ and $L/K$ a finite and purely inseparable extension. 
The integral closure $\mathcal{O}_L$ of $\mathcal{O}_K$ inside $L$ is then also an excellent and Henselian DVR. 
Suppose that we have a smooth model $\mathcal{Y}$ for $X_L$ over $\mathcal{O}_L$.

\begin{theorem}\label{thm: gr desc}
Assume that $H^0(X,\mathcal{T}_{X/K})=0$ and that $H^1(\mathcal{Y},\mathcal{T}_{\mathcal{Y}/\mathcal{O}_L})$ is torsion free. 
Then, $\mathcal{Y}$ descends uniquely to a smooth model for $X$ over $\mathcal{O}_K$.
\end{theorem}

\begin{proof}
Since $X_L$ comes with a canonical descent datum, we apply Proposition \ref{prop: extending descent} 
with $S=\spec{\mathcal{O}_K}$, $S'=\spec{\mathcal{O}_L}$ and $U=\spec{K}$ to extend this descent datum to $\mathcal{Y}$, 
and then apply Propostion \ref{prop: effective} to deduce that this descent datum is effective.
\end{proof}

Since K3 surfaces do not admit non-zero global vector fields \cite{RS76}, the previous result applies to them:

\begin{corollary} \label{cor: gr desc} 
 Suppose that $X$ is a K3 surface over $K$. 
 Then, any smooth model for $X_L$ descends uniquely to a smooth model for $X$. 
 In particular, $X$ has good reduction over $K$ if and only if $X_L$ has good reduction over $L$.
\end{corollary}

\section{Main results} \label{sec: together}

Using the results of \S\ref{sec: descent}, it is now straightforward to deduce our main results, 
via the following proposition.

\begin{proposition} 
 Suppose that $\mathrm{char}(K)=p>0$ and let $X$ be a K3 surface over $K$. 
 Assume that either of the following conditions holds:
\begin{enumerate}
\item for some prime $\ell\neq p$ the $G_K$-representation $H^2_\et(X_{K^{\rm s}},\Q_\ell)$ is unramified;
\item the $\pn$-module $H^2_\rig(X/\mathcal{R})$ admits a basis of horizontal sections.
\end{enumerate}
Then, $X$ satisfies Assumption $(\star)$ if and only if it satisfies Assumption $(\star^{\rm s})$.
\end{proposition}

\begin{proof}
We clearly have that Assumption $(\star^{\rm s})$ implies Assumption $(\star)$. 

For the converse direction, let $L/K$ be a finite extension over which $X$ admits a Kulikov model. 
Note that either condition of the proposition also holds for the base change $X_L$ and hence, by \cite[Theorem 6.4]{CL16} we know that $X_L$ has good reduction. 
If we let $K\subset K_1\subset  L$ denote the maximal separable subextension, it therefore follows from Corollary \ref{cor: gr desc} 
that $X_{K_1}$ has good reduction. In particular, $X$ admits a Kulikov model over $K_1$ and thus, satisfies Assumption $(\star)$.
\end{proof}

Applying the results of \S\ref{sec: recap}, Theorems \ref{theo: LM18 main equi} and \ref{theo: CLL17 main equi} now follow immediately.

\section{A counter-example}\label{sec: counter}

In this section, we will explain how to modify the example given in \cite[\S7]{LM18} to produce K3 surfaces in equicharacteristic $p>0$ that only admit good reduction after a \emph{non-trivial} finite and unramified extension. Let $p\geq 5$, choose $c\in \F_p^*\setminus (\F_p^*)^2$, $a\in \F_p^*\setminus \left\{\frac{16}{27}\right\}$ and let $F\in \F_p[t][x,y,z,w]$ be the polynomial
\[ F:= w(x^3+y^3+z^3+(t+a(1-t))w^3)+(tz^2+xy+tyz)^2-(c-t)t^2y^2z^2. \]
Define $\mathcal{X}:= V(F) \subset \P^3_{\F_p\pow{t}}$ and let $X=\mathcal{X}_{\F_p\lser{t}}$ be its generic fibre.

\begin{theorem} \label{theo: counter example}
$X$ is a smooth K3 surface over $\F_p\lser{t}$ that admits good reduction over $\F_{p^2}\lser{t}$ but not over $\F_p\lser{t}$.
\end{theorem}

\begin{proof}
Exactly as in \cite[Theorem 7.2]{LM18}, we can verify smoothness of $X$ by considering the subscheme cut out by $F$ in $\P^3_{\F_p[t]}$ and reducing modulo $1-t$. The special fibre of $\mathcal{X}$ is isomorphic to the special fibre of the example $\mathcal{X}(p)$ constructed in \cite{LM18}, it is therefore a singular K3 surface with 6 RDP singularities, two of which are defined over $\F_p$ and the other four over $\F_p[\zeta_3]$. Since $\sqrt{c-t}\in \F_{p^2}\pow{t}$, after extending scalars to $\F_{p^2}\pow{t}$ we can resolve the singularities of $\mathcal{X}$ by blowing up either ideal
\[ \mathcal{I}_{\pm}=(w,tz^2+xy+tyz\pm tyz\sqrt{c-t}).\]
However, these two different ideals give rise to distinct elements of the strictly local Picard group $\mathrm{Pic}(\mathcal{O}^{\rm sh}_{\mathcal{X},\bar{x}})$, which are interchanged by the action of  $\mathrm{Gal}(\F_p\lser{t}^{\rm{s}}/\F_p\lser{t})$. Thus, exactly as in \cite[Theorem 7.2]{LM18}, $X$ cannot have a smooth model over $\F_p\pow{t}$.
\end{proof}

\providecommand{\bysame}{\leavevmode\hbox to3em{\hrulefill}\thinspace}
\providecommand{\MR}{\relax\ifhmode\unskip\space\fi MR }
\providecommand{\MRhref}[2]{%
  \href{http://www.ams.org/mathscinet-getitem?mr=#1}{#2}
}
\providecommand{\href}[2]{#2}


\begin{thebibliography}{CCM13}

\bibitem[Art66]{Art66}
M.~Artin, \emph{On isolated rational singularities of surfaces}, Amer. J. Math.
  \textbf{88} (1966), 129--136, \url{http://dx.doi.org/10.2307/2373050}.

\bibitem[BI70]{BI70}
P.~Berthelot and L.~Illusie, \emph{Classes de {C}hern en cohomologie
  cristalline}, C. R. Acad. Sci. Paris S\'{e}r. A-B 270 \textbf{270} (1970),
  A1750--A1752.

\bibitem[CCM13]{CCM13}
B.~Chiarellotto, A.~Ciccioni, and N.~Mazzari, \emph{Cycle classes and the
  syntomic regulator}, Algebra Number Theory \textbf{7} (2013), no.~3,
  533--566, \url{http://dx.doi.org/10.2140/ant.2013.7.533}.

\bibitem[CL16]{CL16}
B.~Chiarellotto and C.~Lazda, \emph{Combinatorial degenerations of surfaces and
  {C}alabi--{Y}au threefolds}, Algebra Number Theory \textbf{10} (2016),
  no.~10, 2235--2266, \url{http://dx.doi.org/10.2140/ant.2016.10.2235}.

\bibitem[CLL17]{CLL17}
B.~Chiarellotto, C.~Lazda, and C.~Liedtke, \emph{A
  {N}\'eron--{O}gg--{S}hafarevich criterion for {K}3 surfaces}, preprint
  (2017), \url{https://arXiv.org/abs/1701.02945/}, to appear in Proc. London
  Math. Soc.

\bibitem[dJ96]{dJ96}
A.~J. de~Jong, \emph{Smoothness, semi-stability and alterations}, Inst. Hautes
  \'Etudes Sci. Publ. Math. \textbf{83} (1996), no.~1, 51--93,
  \url{http://www.numdam.org/item?id=PMIHES_1996__83__51_0}.

\bibitem[GM87]{GM87}
H.~Gillet and W.~Messing, \emph{Cycle classes and {R}iemann-{R}och for
  crystalline cohomology}, Duke Math. J. \textbf{55} (1987), no.~3, 501--538,
  \url{https://doi.org/10.1215/S0012-7094-87-05527-X}.

\bibitem[Kat89]{Kat89}
K.~Kato, \emph{Logarithmic structures of {F}ontaine-{I}llusie}, Algebraic
  Analysis, Geometry and Number Theory (1989), 191--224.

\bibitem[Ked00]{Ked00}
K.~S. Kedlaya, \emph{Descent theorems for overconvergent {$F$}-crystals}, Ph.D.
  thesis, Massachusetts Institute of Technology, 2000.

\bibitem[Ked04a]{Ked04c}
\bysame, \emph{Full faithfulness for overconvergent {$F$}-isocrystals},
  Geometric aspects of {D}work theory. {V}ol. {I}, {II}, Walter de Gruyter,
  Berlin, 2004, pp.~819--835.

\bibitem[Ked04b]{Ked04a}
\bysame, \emph{A {$p$}-adic local monodromy theorem}, Ann. of Math. (2)
  \textbf{160} (2004), no.~1, 93--184,
  \url{https://doi.org/10.4007/annals.2004.160.93}.

\bibitem[Keu16]{Keu16}
J.~Keum, \emph{Orders of automorphisms of {K}3 surfaces}, Adv. Math.
  \textbf{303} (2016), 39--87, \url{https://doi.org/10.1016/j.aim.2016.08.014}.

\bibitem[Kul77]{Kul77}
V.~S. Kulikov, \emph{Degenerations of {$K3$} surfaces and {E}nriques surfaces},
  Izv. Akad. Nauk SSSR Ser. Mat. \textbf{41} (1977), no.~5, 1008--1042, 1199.

\bibitem[LM18]{LM18}
C.~Liedtke and Y.~Matsumoto, \emph{Good reduction of {K3} surfaces}, Compos.
  Math. \textbf{154} (2018), no.~1, 1--35,
  \url{https://doi.org/10.1112/S0010437X17007400}.

\bibitem[LP16]{LP16}
C.~Lazda and A.~P{\'{a}}l, \emph{{R}igid {C}ohomology over {L}aurent {S}eries
  {F}ields}, {A}lgebra and {A}pplications, vol.~21, Springer, 2016,
  \url{http://dx.doi.org/10.1007/978-3-319-30951-4}.

\bibitem[Mat95]{Mat95}
S.~Matsuda, \emph{Local indices of {$p$}-adic differential operators
  corresponding to {A}rtin-{S}chreier-{W}itt coverings}, Duke Math. J.
  \textbf{77} (1995), no.~3, 607--625,
  \url{https://doi.org/10.1215/S0012-7094-95-07719-9}.

\bibitem[Mat15]{Mat15}
Y.~Matsumoto, \emph{Good reduction criteria for {K}3 surfaces}, Math. Z.
  \textbf{279} (2015), no.~1-2, 241--266,
  \url{http://dx.doi.org/10.1007/s00209-014-1365-8}.

\bibitem[Mau14]{Mau14}
D.~Maulik, \emph{Supersingular {K}3 surfaces for large primes}, Duke
  Mathematical Journal \textbf{163} (2014), no.~13, 2357--2425.

\bibitem[Ogu79]{Ogu79}
A.~Ogus, \emph{Supersingular {$K3$} crystals}, Journ\'{e}es de
  {G}\'{e}om\'{e}trie {A}lg\'{e}brique de {R}ennes ({R}ennes, 1978), {V}ol.
  {II}, Ast\'{e}risque, vol.~64, Soc. Math. France, Paris, 1979, pp.~3--86.

\bibitem[PP81]{PP81}
U.~Persson and H.~Pinkham, \emph{Degeneration of surfaces with trivial
  canonical bundle}, Ann. of Math. (2) \textbf{113} (1981), no.~1, 45--66,
  \url{http://dx.doi.org/10.2307/1971133}.

\bibitem[R{\v{S}}76]{RS76}
A.~N. Rudakov and I.~R. {\v{S}}afarevi\v{c}, \emph{Inseparable morphisms of
  algebraic surfaces}, Izv. Akad. Nauk SSSR Ser. Mat. \textbf{40} (1976),
  no.~6, 1269--1307.

\bibitem[SGA72]{SGA7i}
\emph{Groupes de monodromie en g\'eom\'etrie alg\'ebrique. {I}}, Lecture Notes
  in Mathematics, Vol. 288, Springer-Verlag, Berlin-New York, 1972, S\'eminaire
  de G\'eom\'etrie Alg\'ebrique du Bois-Marie 1967--1969 (SGA 7 I), Dirig\'e
  par A. Grothendieck. Avec la collaboration de M. Raynaud et D. S. Rim.

\bibitem[ST68]{ST68}
J.-P. Serre and J.~Tate, \emph{Good reduction of abelian varieties}, Ann. of
  Math. (2) \textbf{88} (1968), 492--517,
  \url{http://dx.doi.org/10.2307/1970722}.

\bibitem[Tsu99]{Tsu99a}
T.~Tsuji, \emph{Poincar\'e duality for logarithmic crystalline cohomology},
  Compos. Math. \textbf{118} (1999), no.~1, 11--41,
  \url{https://doi.org/10.1023/A:1001020809306}.

\end{thebibliography}
\end{document}